%% file: Molino_s_Conjecture_17-01-11.tex
\documentclass{amsart}

\usepackage[dvips]{graphicx}

\usepackage{pstricks-add}
\usepackage{amsmath,amscd,enumerate,amsthm}
\usepackage{amssymb}
\usepackage{amsfonts}

\newtheorem{theorem}{Theorem}[section]
\newtheorem{proposition}[theorem]{Proposition}
\newtheorem{corollary}[theorem]{Corollary}

\newtheorem*{thmA}{Theorem}
\newtheorem*{thmB}{Main Theorem}

\theoremstyle{definition}
\newtheorem{definition}[theorem]{Definition}
\newtheorem{example}[theorem]{Example}
\theoremstyle{remark}
\newtheorem{remark}[theorem]{Remark}

\numberwithin{equation}{section}

\newcommand{\F}{\ensuremath{\mathcal{F}}}
 
\newcommand{\metric}{\ensuremath{ \mathrm{g} }}

\newcommand{\fol}{\ensuremath{\mathcal{F}}}
\newcommand{\ol}{\overline}
\newcommand{\wt}{\widetilde}
\newcommand{\lra}{\longrightarrow}

\newcommand{\In}{\subseteq}

\newcommand{\RR}{\mathbb{R}}
\newcommand{\DD}{\mathbb{D}}
\newcommand{\pr}{\ensuremath{\mathsf{p}}}
\newcommand{\OO}{\ensuremath{\mathsf{O}}}
\newcommand{\base}{B}

\title[]{Closure of singular foliations: the proof of Molino's conjecture}

\author{Marcos M. Alexandrino}
\author{Marco Radeschi}

\address{Marcos M. Alexandrino \hfill\break\indent 
Universidade de S\~{a}o Paulo\\
Instituto de Matem\'{a}tica e Estat\'{\i}stica, \hfill\break\indent
 Rua do Mat\~{a}o 1010,05508 090 S\~{a}o Paulo, Brazil}
\email{marcosmalex@yahoo.de, malex@ime.usp.br}

\address{ Marco Radeschi\hfill\break\indent 
Mathematisches Institut\\
 WWU M\"unster, Einsteinstr. 62, M\"unster, Germany.}
\email{mrade\_02@uni-muenster.de}

\thanks{The first author was  partially supported by FAPESP. The second author is part of the DFG project SFB 878: Groups, Geometry \& Actions.}



\begin{document}

\maketitle

\begin{abstract}
In this paper we prove the conjecture of Molino that for every singular Riemannian foliation $(M,\fol)$, the partition $\ol{\fol}$ given by the closures of the leaves of $\F$ is again a singular Riemannian foliation.
\end{abstract}

\section{Introduction}
%\footnote{Is it Thorbergsson's 65th birthday? I thought he was much older..}

Given a Riemannian manifold $M$, a singular Riemannian foliation $\fol$ on $M$ is, roughly speaking, a partition of $M$ into smooth connected and locally equidistant submanifolds of possibly varying dimension (the \emph{leaves} of $\fol$), which is spanned by a family of smooth vector fields. The precise definition, given in Section \ref{S:prelim}, was suggested by Molino, by combining the concepts of \emph{transnormal system} of Bolton \cite{Bolton} and of \emph{singular foliation} by Stefan and Sussmann \cite{Sussmann}.

A typical example of a singular Riemannian foliation is the decomposition of a Riemannian manifold $M$ into the orbits of an isometric group action $G$ on $M$. Such a foliation is called \emph{homogeneous}. Another example of foliation is given by the partition of an Euclidean vector bundle $E\to L$, endowed with a metric connection, into the \emph{holonomy tubes} around the zero section (cf. Example \ref{E:foliation-vector-bundle}). Such a foliation, which we call \emph{holonomy foliation}, will be a sort-of prototype in the structural results that will appear later on. Holonomy foliations are in general not homogeneous (the zero section $L$ is always a leaf but in general not a homogeneous manifold), however they are \emph{locally homogeneous}, in the sense that  the infinitesimal foliation at every point of $E$ is homogeneous (cf. Sections \ref{SS:inf-foliation} and \ref{SS:orbit-like}). This construction is related to other important types of foliations, like polar foliations \cite{toeben} or Wilking's dual foliation to the Sharafutdinov projection \cite{Wilking}, see Remark \ref{R:eamples}.
%\footnote{First, the paragraph was too long and it was sidetracking too much. Second, I added the reference to Toeben's paper. Finally, I did not add the reference to my paper with Lytchak, as I don't think it is really related to the topic here.}

In general, the leaves of a singular Riemannian foliation might not be closed, even in the simple cases defined above. In the homogeneous case, consider for example the foliation on the flat torus $T^2$ by parallel lines, of irrational slope. These are non closed orbits, of an isometric $\RR$-action on $T^2$.

Given a (regular) Riemannian foliation $(M,\fol)$ with non-closed leaves, Molino proved that replacing the leaves of $\fol$ with their closure yields a new singular Riemannian foliation $\ol{\fol}$. Moreover, he conjectured that the same result should hold true if one starts with a singular Riemannian foliation, and this has become known, in the last decades, as \emph{Molino's Conjecture}.

Molino proved that the closure $\ol{\fol}$ of a singular Riemannian foliation $(M,\fol)$ is a transnormal system \cite{Molino}, thus leaving to prove that it is a singular foliation as well. Moreover, in \cite{Molino2} he suggested a strategy to prove the conjecture for the case of \emph{orbit-like foliations}, i.e. foliations which, roughly speaking, are locally diffeomorphic to the orbits of some proper isometric group action around each point (cf. Section \ref{SS:orbit-like}). A formal alternative proof in this case can be found in [4]. Molino's conjecture was also proved for \emph{polar foliations} and then \emph{infinitesimally polar foliations} in \cite{Alexandrino-molino-polar} and \cite{Alexandrino-Lytchak}, respectively.

These partial results do not cover every possible foliation. Since the Eighties there are examples of non orbit-like foliations, and in recent years it was shown the existence of a remarkably large class of ``infinitesimal'' foliations that are neither homogeneous nor polar, the so-called Clifford foliations \cite{Radeschi} (these infinitesimal foliations have been shown, however, to have an algebraic nature, cf. \cite{Lytchak-Radeschi}). Therefore, it is important to give a complete answer to the conjecture, to fully understand the semi-local dynamic of singular Riemannian foliations.
\\

The goal of this paper is to prove the full Molino's conjecture.

\begin{thmA}(Molino's Conjecture)
Let $(M,\fol)$ be a singular Riemannian foliation on a complete manifold $M$, and let $\ol{\fol}=\{\ol{L}\mid L\in \fol\}$ be the partition of $M$ into the closures of the leaves of $\fol$. Then $(M,\ol{\fol})$ is a singular Riemannian foliation.
\end{thmA}

This result is in fact a direct consequence of the following.

\begin{thmB}
Let $(M,\fol)$ be a singular Riemannian foliation, let $L$ be a (possibly not closed) leaf, and let $U$ be an $\epsilon$-neighbourhood around the closure of $L$. Then for $\epsilon$ small enough, there is a metric $\metric^\ell$ on $U$ and a singular foliation $\widehat{\fol}^\ell$, such that:
\begin{enumerate}
\item $(U,\metric^\ell,\widehat{\fol}^\ell)$ is an orbit-like singular Riemannian foliation.
\item The foliation $\widehat{\fol}^\ell$ coincides with $\fol$ on $\ol{L}$.
\item The closure of $\widehat{\fol}^\ell$ is contained in the closure of $\fol$.
\end{enumerate}
\end{thmB}

In short, the foliation $\widehat{\fol}^\ell$ is obtained by first constructing the \emph{linearized foliation} $\fol^\ell$ of $\fol$ in $U$, which is a subfoliation of $\fol$ spanned by the first order approximations, around $L$, of the vector fields tangent to $\fol$ (see Section \ref{SS:lin-foliation} for a precise definition). The foliation $\widehat{\fol}^\ell$ is then obtained from $\fol^\ell$ by taking the ``local closure'' of the leaves of $\fol^\ell$. The foliations $\fol,\,\fol^\ell,\,\widehat{\fol}^\ell$, together with their closures, are then related by the following inclusions:
\smallskip

\begin{center}
$\renewcommand{\arraystretch}{1.3}
\begin{array}[c]{ccccc}
\vspace{-0.1in} \fol&\supseteq &\fol^\ell&\subseteq&\widehat{\fol}^\ell\\
\vspace{0.1in}\rotatebox{-90}{$\subseteq$}&&\rotatebox{-90}{$\subseteq$}&&\rotatebox{-90}{$\subseteq$}\\
\ol{\fol}&\supseteq &\ol{\fol}^\ell&=&\ol{\widehat{\fol}}{}^\ell\\
\end{array}$
\end{center}

\begin{example}
\label{E:prototype}
Consider an Euclidean vector bundle $E$ over a complete Riemannian manifold $L$, with a metric  connection $\nabla^{E}$ and a connection metric $g^{E}$ (cf. Example \ref{E:foliation-vector-bundle}). Let $H_p$ denote the holonomy group of $(E,\nabla^E)$ at $p$, acting by isometries on the Euclidean fiber $E_p$, and let $(E_p,\F^0_p)$ be a singular Riemannian foliation preserved by the $H_p$-action. Finally, let $K_p$ be the maximal connected group of isometries of $E_{p}$ that
fixes each leaf of $\F^{0}_{p}$ as a set.  

Letting $\F$ the partition of $E$ into the holonomy translates of the leaves of $\F^0_p$ (i.e., for every leaf $\mathcal{L}\in \F^0_p$, $L_{\mathcal{L}}$ denotes the set of points in $E$ that can be reached via $\nabla^E$-parallel translation from a point in $\mathcal{L}$), then $\F$ is a singular Riemannian foliation. In this case, the linearized foliation $\fol^\ell$ is the foliation by the holonomy translates of the $K_p$-orbits in $E_p$, and the local closure of  $\widehat{\fol}^\ell$ is the foliation by the holonomy translates of the $\overline{K}_p$-orbits in $E_p$, where $\overline{K}_p$ denotes the closure of $K_p$ in $\OO(E_p)$.

This can be restated in the language of groupoids: defining $H$ the holonomy groupoid of the connection $\nabla^E$, then $\F=\{H(\mathcal{L})\mid \mathcal{L}\in \F^0_p\}$, $\F^\ell$ is given by the orbits of $HK_p$, and its local closure $\hat{\F}^\ell$ is given by the orbits of $H\overline{K}_p$. 
\end{example}

This paper is organized as follows: after a section of preliminaries (Section \ref{S:prelim}) we show how the Molino's Conjecture follows from the Main Theorem (Section \ref{S:Molino}). In Section  \ref{S:setup} we fix the setup in which we work for the rest of the paper. In Section \ref{S:distributions} we define three distributions of the tangent bundle $TU$. We first use these to obtain information on the local structure of $\fol^\ell$ and define the local closure $\widehat{\fol}^\ell$ (Section \ref{S:linearized}) and then to define the metric $\metric^\ell$ used in the Main Theorem (Section \ref{S:proof}). In this final section we also prove the Main Theorem.

\section*{Acknowledgements}
The authors thank Prof. Lytchak and Prof. Thorbergsson for consistent support.
%\footnote{I am not sure about the part ``The first author dedicates this paper to Prof. Gudlaugur Thorbergsson on the occasion of his 65th Birthday.'', either we both do or neither does\ldots }

\section{Preliminaries}\label{S:prelim}
Given a Riemannian manifold $(M,\metric)$, a partition $\fol$ of $M$ into complete connected submanifolds (the \emph{leaves} of $\fol$) is called a \emph{transnormal system} if geodesics starting perpendicular to a leaf stay perpendicular to all leaves, and a \emph{singular foliation} if every vector tangent to a leaf can be locally extended to a vector field everywhere tangent to the leaves.
%\footnote{I changed the definition to the one we will actually use.}
%the leaves are spanned by a family of smooth vector fields on $M$.
A singular Riemannian foliation will be denoted by the triple $(M,\metric,\fol)$. However, if the Riemannian metric of $M$ is understood, we will drop it and simply write $(M,\fol)$.

The following notation will be used throughout the rest of the paper. Given a point $p\in M$, the leaf of $\fol$ through $p$ will be denoted by $L_p$. A small relatively compact open subset $P\subset L$ is called a \emph{plaque}. The tangent and normal spaces to $L_p$ at $p$ are denoted by $T_pL_p$ and $\nu_pL_p$, respectively. Given some $\epsilon>0$, $\nu^\epsilon_pL_p$ denotes the set of vectors $x\in \nu_pL_p$ with norm $<\epsilon$. If $\epsilon$ is small enough that the normal exponential map $\exp:\nu_p^\epsilon L_p\to M$ is a diffeomorphism onto the image, such image is called a \emph{slice} of $L_p$ at $p$, and it is denoted by $S_p$. The \emph{slice foliation} $\fol|_{S_p}$ denotes the partition of $S_p$ into the connected components of the intersections $L\cap S_p$, where $L\in \fol$.

\subsection{Vector fields of a singular Riemannian foliation}
We review here the main notations about vector fields of a singular Riemannian foliation.

A vector field $V$ is called \emph{vertical} if it is tangent to the leaves at each point. The set of smooth vertical vector fields is a Lie algebra, which is denoted by $\mathfrak{X}(M,\fol)$. 

A vector field $X$ is called \emph{foliated} if its flow takes leaves to leaves or, equivalently, if $[X, V]\in \mathfrak{X}(M,\fol)$ for every $V\in \mathfrak{X}(M,\fol)$. Any vertical vector field is foliated, but there are other foliated vector fields. A vector field is called \emph{basic} if it is both foliated and everywhere normal to the leaves.

\subsection{Homothetic Transformation Lemma}

One of the most fundamental results in the theory of singular Riemannian foliations is the Homothetic Transformation Lemma. A deeper discussion of this lemma, with proof and applications, can be found in Molino \cite{Molino}, Ch. 6, in particular Lemma 6.1 and Proposition 6.7.

Let $(M,\fol)$ be a singular foliation, let $L$ be a leaf of $\fol$, and let $P\subset L$ a plaque. Let $\epsilon>0$ be such that the normal exponential map $\exp:\nu^\epsilon P\to M$ is a diffeomorphism onto its image $B_{\epsilon}(P)$. For any two radii $r_1, r_2=\lambda r_1$ in $(0,\epsilon)$, it makes sense to define the \emph{homothetic transformation}
\[
h_{\lambda}:B_{r_1}(P)\to B_{r_2}(P),\qquad h_{\lambda}(\exp v)=\exp \lambda v.
\]
The leaves of $\fol$ intersect $B_{r_i}(P)$, $i=1,2$, in plaques that foliate $B_{r_i}(P)$. We call $\fol|_{B_{r_i}(P)}$ the foliation of $B_{r_i}(P)$ into the path components of such intersections. One has then the following:

\begin{theorem}[Homothetic Transformation Lemma]
The homothetic transformation $h_{\lambda}$ takes the leaves of $(B_{r_1}, \fol|_{B_{r_1}(P)})$ onto the leaves of $(B_{r_2}, \fol|_{B_{r_2}(P)})$.
\end{theorem}

This result still holds, more generally, if we replace the plaque $P$ by an open subset $\base$ of some submanifold $N\subset M$ which is a union of leaves of the same dimension. In this case we consider some $\epsilon>0$ such that $\exp:\nu^\epsilon \base\to M$ is a diffeomorphism onto the image $B_{\epsilon}(\base)$, and define the homothetic transformation \emph{around $W$}, $h_{\lambda}:B_{r}(\base)\to B_{\lambda r}(\base)$, as before. In this case, an analogous version of the Homothetic Transformation Lemma applies.

%\begin{example}\label{E:foliated-linear}
%As a special example that will appear later on, suppose $(V,\fol)$ is an infinitesimal foliation. Since $\{0\}$ is a leaf, homothetic transformations around $0$ are precisely the rescalings $v\mapsto \lambda v$, $v\in V$, which in this case are global transformations. Moreover, if $W\subset V$ is an \emph{$\fol$-invariant subspace}, i.e. a linear subspace saturated by the leaves of $\fol$, and $V=W\oplus W^{\perp}$ denotes the orthogonal splitting, then homothetic transformations are of the form $(w_1,w_2)\mapsto (w_1,\lambda w_2)$. Together with the homothetic rescalings $v\mapsto \lambda v$, it follows that there is a 2-parameter group of global, linear maps $(w_1,w_2)\mapsto (\lambda w_1, \mu w_2)$ sending leaves of $\fol$ to leaves. In particular, it follows that $W^\perp$ is also $\fol$-invariant.
%\end{example}

\subsection{Infinitesimal foliation}\label{SS:inf-foliation}
Let $(M,\fol)$ be a singular Riemannian foliation, $p\in M$ a point, and $S_p$ a slice at $p$. 

\begin{definition}[Infinitesimal foliation at $p$]
The \emph{infinitesimal foliation of $\fol$ at $p$}, denoted by $(\nu_pL_p,\fol_p)$ is defined as the partition of $\nu_pL_p$ whose leaf at $v\in \nu_pL_p$ is given by
\[
L_v=\{w\in \nu_pL_p\mid \exp_ptw\in L_{\exp_ptv}\quad \forall t>0\textrm{ small enough}\},
\]
where $L_{\exp_ptv}$ denotes the leaf of $(S_p,\fol|_{S_p})$ through $\exp_ptv$.
\end{definition}
The leaf $L_v$ is well defined because, by the Homothetic Transformation Lemma, if $\exp_p{t_0w}$ belongs to the same leaf of $\exp_pt_0v$ for some small $t_0$, then $\exp_ptw$ belongs to the same leaf of $\exp_ptv$ for every $t\in (0,t_0)$. In the following proposition we collect the important facts about infinitesimal foliations that we will need.

\begin{theorem}
Given a singular Riemannian foliation $(M,\fol)$ and a point $p\in M$ with infinitesimal foliation $(\nu_pL_p,\fol_p)$, then:
\begin{enumerate}
\item The foliation $(\nu_pL_p,\fol_p)$  is a singular Riemannian foliation with respect to the flat metric $\metric_p$ at $p$.
\item The normal exponential map $\exp_p:\nu_p^{\epsilon}L_p\to M$ sends the leaves of $\fol_p$ to the leaves of $(S_p,\fol|_{S_p})$.
\item $(\nu_pL_p,\fol_p)$ is invariant under rescalings $r_{\lambda}:\nu_pL_p\to \nu_pL_p$, $r_{\lambda}(v)=\lambda v$.
\end{enumerate}
\end{theorem}
\begin{proof}
1) \cite{Molino}, Prop. 6.5.\\
2) Follows from the definition of infinitesimal foliation, and of slice foliation.\\
3) Via the exponential map $\exp:\nu^{\epsilon}L_p\to S_p$, this corresponds to the Homothetic Transformation Lemma on $S_p$.
\end{proof}

The following fact will come very useful.

\begin{proposition}\label{P:differential-foliated}
%Given a local foliated diffeomorphism $\phi:(M,\fol)\to (M',\fol')$ sending a point $p\in M$ to $p'\in M'$, the differential of $\phi$ induces a linear, foliated isomorphism $\phi_*:(\nu_pL_p,\fol_p)\to (\nu_{p'}L_{p'},\fol'_{p'})$.
Given singular Riemannian foliations $(M,\fol)$, $(M',\fol')$ and a foliated diffeomorphism $\phi: U\to U'$, between open sets $U,U'$ of $M,M'$ respectively, sending a point $p\in U$ to $p'\in U'$, the differential of $\phi$ induces a linear, foliated isomorphism $\phi_*:(\nu_pL_p,\fol_p)\to (\nu_{p'}L_{p'},\fol'_{p'})$.
\end{proposition}
\begin{proof}
By substituting $(M',\metric',\fol')$ with $(M,\phi^*g',\fol)$, the problem can be reduced to the case where $M=M'$, $\phi=id$, $p=p'$, and $\fol=\fol'$ is a singular Riemannian foliation with respect to two metrics, $\metric$ and $\tilde{\metric}$. In the following, we will denote with a ``tilde'' ($\;\tilde{ }\;$) every geometric object related to the metric $\tilde{\metric}$, and without the tilde any geometric object related to $\metric$.

Let $S_p$ (resp. $\wt{S}_p$) denote a slice at $p$ with respect to $\metric$ (resp. $\tilde{\metric}$). Consider the set $\{X_1,\ldots X_k\}\subset \mathfrak{X}(M,\fol)$, $k=\dim L_p$,  of vector fields such that $\{X_1(p),\ldots, X_k(p)\}$ is a basis of $T_pL_p$. Denote by $\Phi_i^t$ the flow of $X_i$, and define $\Phi_{(t_1,\ldots t_k)}=\Phi_k^{t_k}\circ\ldots \circ \Phi_1^{t_1}$.

Around $p$, both $S_p$ and $\wt{S}_p$ are transverse to $\textrm{span}(X_1,\ldots X_k)$ and, up to possibly replacing $S_p$ and $\wt{S}_p$ with smaller open subsets, we can assume that for every $q\in S_p$ there exists a unique $\tilde{q}\in \wt{S}_p$ of the form $\tilde{q}=\Phi_{(t_1,\ldots t_k)}(q)$. This gives rise to a map $H:S_p\to \wt{S}_p$, $H(q)=\tilde{q}$ which is differentiable and, since $q$ and $\tilde{q}$ belong to the same leaf of $\fol$, sends the leaves of $\fol|_{S_p}$ to the leaves of $\fol|_{\wt{S}_p}$. In other words, there is a foliated diffeomorphism $H: (S_p,\fol|_{S_p})\to (\wt{S}_p,\fol|_{\wt{S}_p})$.

Consider the composition $\psi$ of foliated diffeomorphisms
\[
(\nu_p^\epsilon L_p,\fol_p)\stackrel{\exp_p}{\lra}(S_p,\fol|_{S_p})\stackrel{H}{\lra}(\wt{S}_p,\fol|_{\wt{S}_p})\stackrel{\tilde{\exp}_p^{-1}}{\lra} (\tilde{\nu}_p^\epsilon L_p,\wt{\fol}_p).
\]
For any $\lambda\in (0,1)$, one can define a new foliated diffeomorphism
\[
\psi_\lambda:(\nu_p^{\epsilon/\lambda}L_p,\fol_p)\to (\tilde{\nu}_p^{\epsilon/\lambda}L_p,\wt{\fol}_p), \qquad \psi_\lambda(v)={1\over \lambda}\psi(\lambda v).
\]
As $\lambda\to 0$, the maps $\psi_\lambda$ converge to the differential $d_0\psi$ of $\psi$ at $0$. This is an invertible linear map (in particular a diffeomorphism) and, as a limit of foliated maps, it is itself foliated. Therefore, the map
\[
\phi_*:=d_0\psi:(\nu_pL,\fol_p)\lra (\tilde{\nu}_pL,\widetilde{\fol}_p)
\]
satisfies the statement of the proposition.
\end{proof}

%\begin{remark}\label{R:differential-foliated}
%A foliated diffeomorphism $\phi:M\to M'$, induces also a foliated linear isomorphism $(\nu_pL_p,\fol_p)\lra (\nu_{p'}L_{p'},\fol'_{p'})$, where $p'=\phi(p)$, via the composition
%\[
%(\nu_pL_p,\fol_p)\stackrel{i}{\lra} (T_pM,\fol_p) \stackrel{\phi_*}{\lra} (T_{p'}M',\fol'_{p'})\stackrel{\pi_2}{\lra} (\nu_{p'}L_{p'},\fol'_{p'})
%\]
%where $i:\nu_pL_p\to T_pM=T_pL_p\times \nu_pL_p$ denotes the inclusion $v\mapsto(0,v)$, and $\pi_2:T_{p'}M=T_{p'}L_{p'}\times \nu_{p'}L_{p'}\to \nu_{p'}L_{p'}$ denotes projection onto the second factor.
%\end{remark}
\begin{remark}\label{R:essential-part}
Given a singular Riemannian foliation $(M,\fol)$ and a submanifold $N\subset M$ which is a union of leaves of the same dimension, the infinitesimal foliation at a point $p\in M$ splits as a product $(\nu_p(L_p,N)\times \nu_pN,\{\textrm{pts.}\}\times\fol_p|_{\nu_pN})$, where $\nu_p(L_p,N)=\nu_pL_p\cap T_pN$. In this case, the foliation $(\nu_pN,\fol_p|_{\nu_pN})$ is the ``essential part'' of the infinitesimal foliation $(\nu_pL_p,\fol_p)$. By abuse of notation, we will call the foliation $(\nu_pN,\fol_p|_{\nu_pN})$ \emph{infinitesimal foliation at $p$} as well, and denote it by $\fol_p$.
%\footnote{added this remark, in order to justify our later notation of $\fol_p$ as the foliation in $\nu_p\base$.}
\end{remark}

Given a singular Riemannian foliation $(M,\fol)$ and a point $p\in M$, the infinitesimal foliation $(\nu_pL_p,\fol_p)$ at $p$ contains the origin as a leaf of $\fol_p$. Based on this fact, we make the following definition.
\begin{definition}[Infinitesimal foliation]
An \emph{infinitesimal foliation} is a singular Riemannian foliation $(V,\fol)$ on an Euclidean vector space, with the origin $\{0\}$ being a 0-dimensional leaf.
\end{definition}

\subsection{Homogeneous and orbit like foliations}\label{SS:orbit-like}

A singular Riemannian foliation $(M,\fol)$ is called \emph{homogeneous} (sometimes \emph{Riemannian homogeneous}) if there exists a connected Lie group $G$ acting by isometries on $M$, whose orbits are precisely the leaves of $\fol$. Furthermore, a singular Riemannian foliation $(M,\fol)$ is called \emph{orbit-like} if at every point $p\in M$, the infinitesimal foliation $(\nu_pL_p,\fol_p)$ is closed and homogeneous.

\begin{example}[Holonomy foliations]\label{E:foliation-vector-bundle}
An example of orbit like foliation, which will be useful to keep in mind later on, can be constructed as follows. Consider a Riemannian manifold $L$, and an Euclidean vector bundle $E$ over $L$, that is, a vector bundle over $L$ with an inner product $\langle\,,\,\rangle_p$ on each fiber $E_p$, $p\in L$. Let $\nabla^E$ be a metric connection on $E$, i.e. a connection on $E$ such that, for every vector field $X$ on $L$ and sections $\xi, \eta$ of $E$, one has
\[
X\langle \xi, \eta \rangle=\langle\nabla^E_X\xi, \eta\rangle+\langle \xi, \nabla^E_X\eta\rangle.
\]
Given $(E,\nabla^E)$, there is an induced Riemannian metric  $\metric^E$ on $E$, called \emph{connection metric}. Moreover, $\nabla^E$ induces a \emph{parallel transport} on $E$: given $\xi\in E_p$ and a curve $\gamma:[0,1]\to L$ with $\gamma(0)=p$, there exists a unique lift $\xi(t)$, $t\in [0,1]$ with $\xi(0)=\xi$ such that $\nabla^{E}_{\gamma'(t)}\xi(t)=0$ for every $t\in [0,1]$. On $E$ one can now define a foliation $\fol^E$, by declaring two vectors $\xi, \eta\in E$ in the same leaf if they can be connected to one another via a composition of parallel transports. The leaves of $\fol^E$ are usually referred to as the \emph{holonomy tubes} around the zero section $L\subset E$, and they define a singular Riemannian foliation on $(E,\metric^E)$. Moreover, the infinitesimal foliation at any point of $E$ is homogeneous: in fact, for any point $p$ along the zero section $L$, one can first construct the holonomy group $H_p$ of the connection $\nabla^E$, which acts by isometries on the fiber $E_p$ and whose orbits are precisely the leaves of the infinitesimal foliation of $\fol^E$ at $p$. Similarly, the infinitesimal foliation at a point $\xi\in E_p$ is given by the orbits in $\nu_\xi L_\xi$ of the stabilizer $H_\xi\subset H_p$ of $\xi$. The foliation $\fol^{E}$ coincides with its own linearization with respect to the zero section (see definition in Section \ref{SS:lin-foliation}). 
Moreover, if the leaves of $\fol^E$ are closed then $(E,\metric^E, \fol^E)$ is an orbit-like foliation.
\end{example}

\begin{remark}\label{R:eamples}
%\footnote{I moved the discussion in the introduction, to this remark}
When $L\subset M$ is a submanifold of somewhat special geometry, the holonomy foliation on the normal bundle $E$ of $L$, endowed with the Levi-Civita connection, induces via the normal exponential map a foliation on the whole of $M$. For example, if $L$ has \emph{parallel focal structure}, then the induced foliation on $M$ is a \emph{polar foliation} \cite{toeben}. If $M$ is a complete, non-compact manifold with sectional curvature $\geq 0$ and $L$ is a \emph{soul} of $M$ \cite{Cheeger-Gromoll}, then the induced foliation on $M$ is Wilking's \emph{dual foliation} to the Sharafutdinov projection \cite{Wilking}.
\end{remark}

Although in principle the property of being orbit-like might depend on the metric, the following proposition shows in fact that being orbit like is invariant under foliated diffeomorphisms.

\begin{proposition}\label{P:orbit-like-invariant}
The following hold:
\begin{enumerate}
\item Given a foliated linear isomorphism  $\varphi:(V,\fol)\to(V',\fol')$ between infinitesimal foliations, $(V,\fol)$ is homogeneous if and only if $(V',\fol')$ is homogeneous.
\item Given a foliated diffeomorphism $\phi:(M,\fol)\to (M',\fol')$ between singular Riemannian foliations, $(M,\fol)$ is orbit-like if and only if $(M',\fol')$ is orbit-like.
\end{enumerate}

\end{proposition}
\begin{proof}
1) By the symmetric roles of $V$ and $V'$, it is enough to show that if $(V,\fol)$ is homogeneous, so is $(V',\fol')$. Suppose that $(V,\fol)$ is homogeneous, and therefore the foliation $\fol$ is spanned by Killing fields. Recall that a vector field $X$ on an Euclidean space $(V,\metric)$ is Killing if and only if is of the form $X(v)=Av$, where $A$ is a skew symmetric endomorphism of $V$, in the sense that $\metric(Av,v)=0$ for every $v\in V$. Letting $\{X_1,\ldots X_k\}$ denote a set of Killing fields on $V$ spanning the foliation $\fol$, the set $\{Y_1,\ldots Y_k\}$ with $Y_i(v')=\varphi_*\big(X_i\big(\varphi^{-1}(v')\big)\big)$ spans the foliation $\fol'$ as well. Since $\varphi$ is a linear map and the vector fields $X_i$ are linear, it follows that $Y_i$ can be written as $Y_i(v')=B_iv'$ for some endomorphism $B_i:V'\to V'$, $i=1\ldots k$. Since $(V',\metric', \fol')$ is a singular Riemannian foliation, the leaf $L_{v'}$ through $v'$ lies in a distance sphere from the origin, and in particular $\metric'(T_{v'}L_{v'},v')=0$. Since $Y_i(v')$ is tangent to $L_{v'}$, it follows that
\[
0=\metric'(Y_i(v'),v')=\metric(B_iv',v')
\]
In other words, $B_i$ is skew-symmetric and thus $Y_i$ is a Killing field as well. Therefore the foliation $(V',\fol')$ is spanned by Killing vector fields, hence it is homogeneous as well.

2) Up to exchanging the roles of $M$ and $M'$, it is enough to show that if $(M,\fol)$ is orbit-like, so is $(M',\fol')$. Fixing a point $p\in M$, Proposition  \ref{P:differential-foliated} states that the foliated diffeomorphism $\phi$ induces a foliated linear isomorphism $\phi_*:(\nu_pL_p, \fol_p)\lra (\nu_{p'}L_{p'},\fol_{p'})$, where $p'=\phi(p)$. Since $(M,\fol)$ is orbit-like, it follows that $(\nu_pL_p, \fol_p)$ is closed and homogeneous. From the first point above it follows that $(\nu_{p'}L_{p'},\fol_{p'})$ is homogeneous as well, and by the continuity of $\phi_*$ one has that $(\nu_{p'}L_{p'},\fol_{p'})$ is closed. Since $p'$ was chosen arbitrarily, it follows that $(M',\fol')$ is orbit-like.
\end{proof}

\subsection{Linearization, and linearized foliation}\label{SS:lin-foliation}
Let $(M,\fol)$ be a singular Riemannian foliation, $B\subset M$ a submanifold saturated by leaves, and $U\in M$ an $\epsilon$-tubular neighbourhood of $B$ with metric projection $\pr:U\to B$. Given a vector field $V$ in $U$ tangent to the leaves of $\fol$, it is possible to produce a new vector field $V^\ell$, called the \emph{linearization of $V$ with respect to $B$}, as follows:
\[
V^\ell=\lim_{\lambda\to 0} (h_{\lambda}^{-1})_*(V|_{h_{\lambda}(U)})
\]
where $h_\lambda:U\to U$ denotes the homothetic transformation around $B$. From \cite{Mendes-Radeschi}, Prop. 5, the linearization $V^\ell$ is a smooth vector field invariant under the homothetic transformation $h_\lambda$, and it coincides with $V$ along $B$. %Moreover, the flow of a linearized vector field takes fibers of the foot point projection

%\subsection{Linearized foliation}\label{SS:lin-foliation}
%Let $(M,\fol)$ be a singular Riemannian foliation, $B\subset M$ a submanifold saturated by leaves, and $U\subset M$ an $\epsilon$-tubular neighbourhood of $B$ with metric projection $\pr:U\to B$.
On $U$, consider the module $\mathfrak{X}(U,\fol)^\ell$ given by the linearization, with respect to $B$, of the vector fields in $\mathfrak{X}(U,\fol)$:
\[
\mathfrak{X}(U,\fol)^\ell=\{V^\ell\mid V\in \mathfrak{X}(U,\fol)\}.
\]
Let $\mathsf{D}$ the pseudogroup of local diffeomorphisms of $U$, generated by the flows of linearized vector fields, and let $(U,\fol^\ell)$ the partition of $U$ into the orbits of diffeomorphisms in $\mathsf{D}$. By Sussmann \cite[Thm. 4.1]{Sussmann}, such orbits are (possibly non-complete) smooth submanifolds of $M$. Moreover, as noted  By Molino \cite[Lem. 6.3]{Molino}, this foliation is spanned, at each point, by the vector fields in $\mathfrak{X}^\ell(U,\fol)$.

We call $(U, \fol^\ell)$ the \emph{linearized foliation of $\fol$ with respect to $B$}. We will show, later, that the leaves of the linearized foliation are actually complete, and have a particularly nice local structure (cf. Section \ref{S:linearized}).

Given a point $p\in B$, define $U_p=\pr^{-1}(p)\subset U$ and let $\fol_p$ (resp. $(\fol^\ell)_p$) denote the partition of $U_p$ into the connected components of $L\cap U_p$, as $L$ ranges through the leaves of $\fol$ (resp. $\fol^\ell$). If $U_p$ is given the flat metric $\metric_p$ of $\nu_pB$ via the exponential map $\exp_p:\nu_p^{\epsilon}B\to U_p$, then $\fol_p$ corresponds to the infinitesimal foliation at $p$ (cf. Remark \ref{R:essential-part}) which justifies the notation of $\fol_p$ for this foliation. Furthermore,
%\footnote{added a few words in here}
as noted in \cite[Sec. 6.4]{Molino}, $(\fol^\ell)_p$ is given by the linearization of $(U_p,\metric_p,\fol_p)$ with respect to the origin. In other words, $(\fol^\ell)_p=(\fol_p)^\ell$ and it makes sense to denote this foliation simply by $\fol_p^\ell$. Moreover, letting $\OO(\fol_p)$ denote the Lie group of (linear) isometries of $(U_p,\metric_p)$ sending every leaf to itself, one has:
\begin{proposition}\label{P:linearized-infinitesimal}
The foliation $(U_p,\fol^\ell_p)$ is homogeneous, given by the orbits of the identity component $H_p$ of $\OO(\fol_p)$.
\end{proposition}
\begin{proof}
We identify here $U_p$ with a neighbourhood of the origin in $\nu_pB$ via the exponential map, and we think of $\fol_p^\ell$ as the linearization of $\fol_p$.

Given a vector field $V\in \mathfrak{X}(U_p,\fol_p)$, its linearization $V^\ell$ is linear, in the sense that $V^\ell_p=A\cdot p$ for some $A\in \textrm{End}(U_p)$. Since $\fol_p$ is a singular Riemannian foliation, the leaves are tangent to the distance spheres around the origin and therefore perpendicular to the radial directions from the origin: $\langle V^\ell_p,p\rangle=0$. In other words, $V^\ell_p=A\cdot p$ with $A$ skew symmetric, which implies that the flow of $V^\ell$ is an isometry of $U_p$. Moreover, since $V^\ell$ is everywhere tangent to the leaves of $\fol_p^\ell$, the flow of $V^\ell$ is a 1-parameter group in $H_p$, moving every leaf of $(\fol_p)^\ell$ to itself. In particular, the orbits of $H_p$ are contained in the leaves of $(\fol_p)^\ell$. 

However,  by definition of $H_p$, the tangent space of a $H_p$-orbit through a point $q\in U_p$ is given by
\[
T_q(H_p\cdot q)=\{W_q\mid W\textrm{ Killing vector field tangent to the leaves of }\fol_p\}
\]
and such vector fields coincide precisely with the vector fields in $\mathfrak{X}(U_p,\fol_p)^\ell$. Therefore, $H_p\cdot q$ is the integral manifold of $\mathfrak{X}(U_p,\fol_p)^\ell$ through $q$.
\end{proof}

\section{Molino's conjecture, assuming the Main Theorem}\label{S:Molino}

Before proving the Main Theorem, we show how Molino's Conjecture follows from it as a corollary.

\begin{proof}[Proof of Molino's Conjecture]
Let $(M,\fol)$ be a singular Riemannian foliation, and let $\ol{\fol}$ denote the closure of $\fol$. Molino himself proved that $\ol{\fol}$ is a partition into complete smooth closed submanifolds, and that $\ol{\fol}$ is a transnormal system. Therefore, in order to prove the conjecture, it is enough to show that for any leaf $L\in \fol$ with closure $\ol{L}$ and any vector $v\in \nu(L,\ol{L}):=\nu L\cap T\ol{L}$, there exists a smooth extension of $v$ to a vector field $V$ everywhere tangent to the leaves of $\ol{\fol}$.

Let $U$ be a tubular neighbourhood of $\ol{L}$, and let $(U, \widehat{\fol}^\ell)$ be the foliation satisfying the Main Theorem. Since $\widehat{\fol}^\ell$ coincides with $\fol$ along $\ol{L}$, it follows that $L$ is a leaf of $\widehat{\fol}^\ell$ as well. Since $\widehat{\fol}^\ell$ is an orbit-like foliation, by Theorem 1.6 of \cite{Alexandrino-Radeschi-smooth-flows}, given $v\in \nu(L,\ol{L})$ there is a vector field $V$ extending $v$ which is tangent to the closure of $\widehat{\fol}^\ell$. Since this closure is contained in $\ol{\fol}$, it follows that $V$ is also tangent to $\ol{\fol}$ and this ends the proof of the conjecture.
\end{proof}

\section{The setup}\label{S:setup}

Fix a leaf $L$, and distance tube $U=B_{\epsilon}(L)$ around $\ol{L}$. Using the normal exponential map $\exp:\nu\ol{L}\to M$, $U$ can be identified with the $\epsilon$-tube $\nu^\epsilon \ol{L}$ around the zero section. By the Homothetic Transformation Lemma, the pull-back foliation $\exp^{-1}\fol$ on $\nu^{\epsilon}\ol{L}$ is invariant under the rescalings $r_{\lambda}:\nu^{\epsilon}\ol{L}\to \nu^{\epsilon}\ol{L}$, $r_\lambda(p,v)=(p,\lambda v)$ for any $\lambda\in (0,1)$.

For this reason, in the following sections we will be considering the (slightly more general) setup:
\begin{itemize}
\item $U$ is the $\epsilon$-tube around the zero section of some Euclidean vector bundle $E\to \base$ (in our case $\base=\ol{L}$), with projection $\pr: U\to \base$.
\item $\metric$ is a Riemannian metric on $U$ with the same radial function as the Euclidean metric on each fiber of $E$.
\item $(U,\metric,\fol)$ is a singular Riemannian foliation on $U$, invariant under rescalings $r_\lambda$. In particular, the zero-section $\base$ is saturated by leaves and the projection $\pr$ sends leaves onto leaves.
\item The restriction $\fol_\base=\fol|_\base$ is a regular Riemannian foliation.
\item For every leaf $L\In \base$ and any point $p\in L$, the normal exponential map $\nu_p^{\epsilon}L\to U$ is an embedding.
\end{itemize}

\section{Three distributions}\label{S:distributions}

Let $(U,\metric,\fol)$, $\pr:U\to \base$ be as in Section \ref{S:setup}. In order to prove the Main Theorem, it is first needed to produce a nicer metric on $U$, and for this we first need to split the tangent space of $U$ into three components. The first, $\mathcal{K}=\ker\pr_*$, is the distribution tangent to the fibers of $\pr$. For the remaining two notice that, since the foliation $(\base,\fol_\base)$ is regular, the tangent bundle $T\base$ splits into a tangent and a normal part to the foliation: $T\base=T\fol|_{\base}\oplus\nu\fol|_{\base}$. The last two distributions will be constructed as (appropriately chosen) extensions $\mathcal{T}$ and $\mathcal{N}$ of $T\fol|_{\base}$ and $\nu\fol|_{\base}$ respectively, to the whole of $U$.
\begin{figure}[!htb]
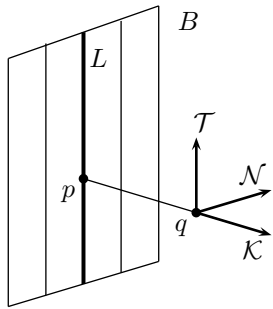

\begin{centering}
\include{3distributions}
\end{centering}
\caption{The distributions at $q$.}
%The space in gray is $\mathcal{S}_q=T_qS_p$.}
\end{figure}

\subsection{The distribution $\mathcal{T}$}

From \cite{Alexandrino-desingularization} there exists a distribution $\widehat{\mathcal{T}}$ of rank $\dim \fol|_\base$, which extends $T\fol|_{\base}$ and is everywhere %transverse to the distribution $\mathcal{S}=\mathcal{K}\oplus \mathcal{N}$  but also everywhere
tangent to the leaves of $\fol$. %Recall from the previous section that $\mathcal{S}$, is the distribution tangent to the slices $S_p$, for $p\in \base$.

The distribution $\mathcal{T}$ is simply defined as the \emph{linearization of $\widehat{\mathcal{T}}$ with respect to $\base$}, as follows: consider a family of vector fields $\{V_\alpha\}_{\alpha}$ spanning $\widehat{\mathcal{T}}$. Since $\widehat{\mathcal{T}}|_{\base}$ is tangent to $\base$, the vector fields $V_\alpha$ lie tangent to $\base$ as well and therefore it makes sense to consider their linearization $V_\alpha^\ell$ with respect to $\base$. By the properties of the linearization, these linearized vector fields still span a smooth distribution  of the same rank as $\widehat{\mathcal{T}}$, which we call $\mathcal{T}$.

\subsection{The distribution $\mathcal{N}$}
At each point $q\in U$ with $\pr(q)=p$, the slice $S_p=\exp_p(\nu_p^{\epsilon}L_p)$ contains $q$ as well as the whole $\pr$-fiber $U_p$ through $p$. In particular, $\mathcal{K}_q$ lies tangent to $S_p$. Moreover, $S_p$ comes equipped with a flat metric $\metric_p$, inherited from the metric on $\nu_pL$ via the diffeomorphism $\exp_p:\nu^\epsilon L_p\to S_p$.

Define $\widehat{\mathcal{N}}_q$ as the subspace of $T_qS_p$ which is $\metric_p$-orthogonal to $\mathcal{K}_q$.  Finally, define $\mathcal{N}$ as the \emph{linearization} of $\widehat{\mathcal{N}}$, as defined in the previous section.
%follows: consider a family of vector fields $\{X_\alpha\}_{\alpha}$ spanning $\widehat{\mathcal{N}}$. Since $\widehat{\mathcal{N}}|_{\base}$ is tangent to $\base$, the vector fields $X_\alpha$ lie tangent to $\base$ as well and therefore it makes sense to consider their linearizations $X_\alpha^\ell$ with respect to $\base$. By the properties of the linearization, these linearized vector fields still span a smooth distribution  of the same rank as $\widehat{\mathcal{N}}$, which we call $\mathcal{N}$.

The distributions $\widehat{\mathcal N}$ and $\mathcal N$ satisfy the following property:
\begin{proposition}\label{P:extension}
For every smooth $\fol_\base$-basic vector field $X_0$ along a plaque $P$ in $\base$ there exists a smooth extension $X$ to an open set of $U$ such that
\begin{enumerate}
\item $X$ is foliated and tangent to $\widehat{\mathcal N}$.
\item The linearization $X^{\ell}$ of $X$ with respect to $\base$ is tangent to $\mathcal{N}$, and it is foliated with respect to both $\fol$ and $\fol^\ell$.
\end{enumerate}
\end{proposition}
\begin{proof}
1) Fix a leaf $L$ in $\base$, a plaque $P\subset L$ and a parametrization
\[
\varphi:(-1,1)^k\to P\subset L,
\]
where $k=\dim \fol_B$. We first show that there exists a small neighbourhood of $P$ in $U$, on which any $\fol_B$-basic vector field $X_0$ along $P$ can be extended to a foliated vector field $X'$, \emph{whose restriction to $\pr^{-1}(P)$ is tangent to $\widehat{\mathcal{N}}$}.

Let $\partial_{y_1},\ldots, \partial_{y_k}$ be coordinate vector fields on $P$, and let $Y_1, \ldots Y_k$ denote vector fields, linearized \emph{with respect to $L$}, that extend $\partial_{y_1},\ldots \partial_{y_k}$ to a neighbourhood of $P$ in $U$.  There is a foliated diffeomorphism
\begin{align*}
F: (P\times S_p, P\times \fol_{S_p})&\lra (U,\fol)\\ 
(\varphi(y_1,\ldots y_k),q)&\longmapsto \Phi_k^{y_k}\circ\ldots\circ \Phi_1^{y_1}(q)
\end{align*}
where $\Phi_i^{y_i}$ is the flow of $Y_i$, after time $y_i$.

%By possibly shrinking $\epsilon$, we can replace $(S_p,\fol_{S_p})$ with an open subset, which splits as
%\[
%(\nu^\epsilon_p(L,\base),\{\textrm{pts.}\})\times(\nu_p^\epsilon\base, \fol|_{\nu_p^\epsilon\base})
%\]
%where $\nu^\epsilon_p(L,\base)=\nu_p^\epsilon L\cap T_p\base$. Moreover, if $S_p$ is endowed with the Euclidean metric $\metric_p$ on $\nu_pL$, this splitting is in fact Riemannian.

Furthermore,
%\footnote{I modified this paragraph, to try making it clearer}
the foliation $(P\times S_p,P\times\fol_{S_p})$ locally splits as
\[
(P\times \nu_p(L,\base)\times \nu_p\base,P\times \{\textrm{pts.}\}\times \fol|_{\nu_p\base})
\]
where $\nu(L,\base)=\nu L \cap T\base$. Moreover, if $S_p$ is endowed with the Euclidean metric $\metric_p$ on $\nu_pL$, the splitting $S_p=\nu_p(L,\base)\times \nu_p\base$ is in fact Riemannian.

The map $F$ satisfies the following:
\begin{itemize}
\item The set $P\times \{0\}\times \nu_p^\epsilon\base$ is sent to $\pr^{-1}(P)=\nu^\epsilon\base|_P$.
\item The set $P\times \nu_p^{\epsilon}(L,\base)\times \{0\}$ is sent to a neighbourhood of $P$ in $B$.
\item Since $F$ is defined via linearized vector fields, each fiber $\{p'\}\times S_p\subset P\times S_p$ is sent, via $F$, to the slice $S_{p'}$, isometrically with respect to the flat metrics on $S_p$ and $S_{p'}$ (cf. \cite{Mendes-Radeschi}).
%\footnote{added a few words to try making it clearer.}
\end{itemize}

From the last point, it follows that the distribution of $P\times \nu_p(L,\base)\times \nu_p\base$ tangent to the second factor is sent, along $\nu^\epsilon\base|_{P}$, precisely to the distribution $\widehat{\mathcal{N}}$.

Any $\fol_B$-basic vector field $X_0$ along $P$ corresponds, via $F$, to a vector field along $P\times \{0\}\times \{0\}$ of the form $(0,x_0,0)$ where $x_0\in \nu_p(L,\base)$ is a fixed vector. One can clearly extend such a vector field to the foliated vector field $X'=F_*(0,x_0,0)$. Since $F$ is a foliated map, the vector field $X'$ is a foliated vector field, whose restriction to $B$ is tangent to $B$  by the second point above. Moreover, by the discussion above the restriction of $X'$ to $\pr^{-1}(P)$ is tangent to $\widehat{\mathcal{N}}$.

This proves the first claim, made at the beginning of the proof. In particular, since the plaque $P$ was chosen arbitrarily, this shows that the distribution $\widehat{\mathcal{N}}$ is \emph{foliated}: that is, given a vector $y$ tangent to $\widehat{\mathcal{N}}$ at a point $q$, there exists a foliated extension $Y$ along a plaque containing $q$ which is everywhere tangent to $\widehat{\mathcal{N}}$. It is easy to see that $\mathcal{K}$ and ${\mathcal{T}}$ are foliated as well. In particular, given the foliated vector field $X'$, the (unique) decomposition
\[
X'=X'_{_\mathcal{K}}+X'_{_\mathcal{T}}+X'_{_{\widehat{\mathcal{N}}}},\qquad X'_{_\mathcal{K}}\in \mathcal{K}, \quad X'_{_\mathcal{T}}\in \mathcal{T},\quad X'_{_{\widehat{\mathcal{N}}}}\in\widehat{\mathcal{N}}
\]
produces three vector fields $X'_{_\mathcal{K}}, \, X'_{_\mathcal{T}},\,X'_{_{\widehat{\mathcal{N}}}}$ which are foliated. In particular, the vector field $X=(X')_{\widehat{\mathcal{N}}}$ is foliated, everywhere tangent to $\widehat{\mathcal{N}}$, and it extends $X_0=(X_0)_{_{\widehat{\mathcal{N}}}}$ to an open set of $U$, as we needed to show.

2) Since $X$ is tangent to $\widehat{\mathcal{N}}$, its linearization $X^\ell$ is tangent to the linearization of $\widehat{\mathcal{N}}$, which is $\mathcal{N}$. Moreover, since $X$ is foliated and $r_\lambda:U\to U$ is a foliated map, $X^\ell=\lim_{\lambda\to 0}(r_\lambda^{-1})_*X\circ r_\lambda$ is foliated as well. Finally, since $X$ is foliated, for every vector field $V$ tangent to $\fol$ one has that $[X,V]$ is also tangent to $\fol$. Since $r_\lambda$ is a diffeomorphism, one computes
\begin{align*}
[X^\ell,V^\ell]	&=\lim_{\lambda\to 0}[(r_\lambda^{-1})_*X\circ r_\lambda,(r_\lambda^{-1})_*V\circ r_\lambda]\\
				&=\lim_{\lambda\to 0}(r_\lambda^{-1})_*[X,V]\circ r_\lambda\\
				&=[X,V]^\ell.
\end{align*}
Since the linearization $V^\ell$ are precisely the vector fields generating $\fol^\ell$, it follows from the equation above that $[X^\ell,V^\ell]$ is tangent to $\fol^\ell$ whenever $V^\ell$ is, and therefore $X^\ell$ is foliated with respect to $\fol^\ell$.
\end{proof}

\section{Structure of $\fol^\ell$, and the local closure $\widehat{\fol}^\ell$}\label{S:linearized}

Using the extensions $X^\ell$ defined in Proposition \ref{P:extension}, one can prove the following:

\begin{proposition}\label{P:local-form}
Around any point $p\in \base$ there is a neighbourhood $W$ of $p$ in $B$ such that $(\pr^{-1}(W),{\fol}^\ell|_{\pr^{-1}(W)})$ is foliated diffeomorphic to a product
\[
(\DD^k\times \DD^{m-k}\times U_p, \DD^k\times\{\textrm{pts.}\}\times {\fol}^\ell_p)
\]
where $k=\dim \fol|_\base$ and $m=\dim \base$.
%, $\DD^k$ is foliated by one leaf, and $\DD^{m-k}$ is foliated by points. Moreover, along $U_p$ the distributions $\mathcal{T}, \mathcal{N}$ and $\mathcal{K}$ are tangent to the first, second and third factor, respectively.
%In particular, ${\fol}^\ell$ is a singular foliation.
\end{proposition}
\begin{proof}
Let $W$ be a coordinate neighbourhood of $\base$ around $p$, with a foliated diffeomorphism $\varphi:(W,\fol|_W)\to (\DD^k\times \DD^{m-k},\DD^k\times\{pts.\})$. Let ${\partial \over \partial y_1}, \ldots {\partial \over \partial y_k}$ denote a basis of vector fields in $W$ tangent to the leaves of $\fol|_W$, and let $V_1,\ldots V_k$ denote vector fields on $\pr^{-1}(W)$, linearized with respect to $\base$, extending ${\partial \over \partial y_i}$, $i=1,\ldots k$ and spanning the foliation $\mathcal{T}$. Similarly, let ${\partial \over \partial x_1},\ldots {\partial \over \partial x_{m-k}}$ denote a basis of basic vector fields in $W$ normal to the leaves, and let $X^\ell_1,\ldots X^\ell_{m-k}$ denote linearized vector fields in $\pi^{-1}(W)$ defined as in Proposition \ref{P:extension}, extending the vectors ${\partial \over \partial x_i}$, $i=1,\ldots m-k$. Finally, define $\Phi_i^t$ and $\Psi_i^t$ the flows of $V_i$ and $X^\ell_i$ respectively, after time $t$, and let
\begin{align*}
G:\DD^k\times \DD^{m-k}\times U_p&\longrightarrow \pr^{-1}(W)\\
((t_1,\ldots,t_k),(s_1,\ldots, s_{m-k}),q)&\longmapsto \Phi_k^{t_k}\circ\ldots\circ \Phi_1^{t_1}\circ \Psi_{m-k}^{s_{m-k}}\circ\ldots \circ \Psi_1^{s_1}(q).
\end{align*}
Since the $V_i$ and $X^\ell_i$ are linearized, they take fibers of $\DD^k\times \DD^{m-k}\times U_p\to \DD^k\times \DD^{m-k}$ to fibers of $\pr:\pr^{-1}(W)\to W$. Since the flows $\Psi_i$ send the leaves of $\fol^\ell$ to leaves, and the flows $\Phi_i$ take leaves of $\fol^\ell$ to themselves, the leaves of $\DD^k\times (\DD^{m-k},\{\textrm{pts.}\})\times (U_p, \fol^\ell_p)$ are sent into the leaves of $\fol^\ell$. Since the differential $dG$ is invertible at $(0,0,p)\in\DD^k\times \DD^{m-k}\times U_p$, it is a diffeomorphism around $G(0,0,p)=p$ and, by dimensional reasons, the leaves of $(\DD^k\times \DD^{m-k}\times U_p, \DD^k\times\{\textrm{pts.}\}\times \fol^\ell_p)$ are mapped diffeomorphically onto the leaves of $(\pr^{-1}(W),\fol^\ell|_{\pr^{-1}(W)})$.

%first statement.
%
%For the second statement, notice that $G$ takes $\{0\}\times\{0\}\times U_p$ to $U_p$, and that for any $q\in U_p$, $v\in T_qU_p$, by construction one has:
%\begin{itemize}
%\item $d_{(0,0,q)}G\left({\partial\over\partial t_i}\right)=V_i$ is tangent to $\mathcal{T}_q$.
%\item $d_{(0,0,q)}G\left({\partial\over\partial s_i}\right)=X_i^\ell$ is tangent to $\mathcal{N}_q$.
%\item $d_{(0,0,q)}G(v)=v$ is tangent to $\mathcal{K}_q$.
%\end{itemize}
\end{proof}

%In particular, $G$ takes leaf closures diffeomorphically to leaf closures, and in particular it is a foliated diffeomorphism
%$$
%\DD^k\times (\DD^{m-k},\{\textrm{pts.}\})\times (U_p, \ol{H}_p\textrm{-orbits})\longrightarrow (\pr^{-1}(W), \widehat{\fol}|_{\pr^{-1}(W)}).
%$$
%\end{proof}

\subsection*{The local closure of $\fol^\ell$}
Even though $(U_p,\fol^\ell_p)$ is homogeneous for every $p\in \base$, it might be the case that its leaves are not closed, which happens when the group $H_p\In \OO(U_p)$ defined in Proposition \ref{P:linearized-infinitesimal} is not closed. To obviate this problem we define a new foliation $\widehat{\fol}^\ell$, called the \emph{local closure} of $\fol^\ell$, such that $\fol^\ell\subset \widehat{\fol}^\ell\subset \ol{\fol^\ell}$ and whose restriction $\widehat{\fol}^\ell_p$ to each $\pr$-fiber $U_p$ is homogeneous and closed.

Recall that $\fol^\ell$ is defined by the orbits of the pseudogroup $\mathsf{D}$ of local diffeomorphisms, generated by the flows of linearized vector fields. For each $q\in U_p$, consider  the closure $\ol{H}_p$ of $H_p$ in $\OO(U_p)$, and define the $\widehat{\fol}^\ell$-leaf $\widehat{L}_q$ through $q$ to be the $\mathsf{D}$-orbit of $\ol{H}_p\cdot q$:
\[
\widehat{L}_q=\{q'=\Phi(h\cdot q)\mid \Phi\in \mathsf{D},\, h\in \ol{H}_p\}
\]
Let $\sim$ denote the relation $q\sim q'$ if and only if $q'=\Phi(h\cdot q)$ for some $\Phi\in \mathsf{D}$ and $h\in \ol{H}_p$. In this way, the leaf of $\widehat{\fol}^\ell$ through $q$ can be rewritten as $\{q'\in U\mid q'\sim q\}$. As for the other foliations, for every $p\in \base$ we define $(U_p,\widehat{\fol}^\ell_p)$ to be the partition of $U_p$ into the connected components, of the intersections of $U_p$ with the leaves in $\widehat{\fol}^\ell$.
%\footnote{added this sentence.}

\begin{proposition}\label{P:prop-loc-closure}
The following hold:
\begin{enumerate}
\item $\widehat{\fol}^\ell$ is a well defined partition of $U$.
\item For every $p\in \base$ the leaves of $\widehat{\fol}^\ell_p$ are the orbits of $\ol{H}_p$ on $U_p$.
%\item $\widehat{\fol}^\ell$ is a singular foliation.
%\item For every $p\in \base$, $(U_p,\widehat{\fol}^\ell_p)$ is a homogeneous singular Riemannian foliation with closed orbits.
%\item The leaves of $\widehat{\fol}^\ell$ are complete.
\end{enumerate}
\end{proposition}
\begin{proof}
1. One must prove that the relation $\sim$ defined above is an equivalence relation. For this, notice that, since any $\Phi\in \mathsf{D}$ defines a foliated isometry between $(U_p,\fol^\ell_p)$ and $(U_{\Phi(p)},\fol^\ell_{\Phi(p)})$ for any $p\in \base$, in particular it defines a foliated isometry between the respective closures $(U_p,\ol{H}_p)$ and $(U_{\Phi(p)},\ol{H}_{\Phi(p)})$. In particular, for any $h\in \ol{H}_p$ and $\Phi\in \mathsf{D}$, one has $h'=\Phi\circ h\circ \Phi^{-1}\in \ol{H}_{\Phi(p)}$.

- Reflexivity of $\sim$: if $q'\sim q$ then $q'=\Phi(h(q))$ for some $h\in \ol{H}_p$ and $\Phi\in \mathsf{D}$. Then $q'=h'(\Phi(q))$, where $h'=\Phi\circ h\circ \Phi^{-1}\in \ol{H}_{\Phi(p)}$, and therefore $q=\Phi^{-1}((h')^{-1}q')$, that means $q\sim q'$.

- Transitivity of $\sim$: if $q'\sim q$ and $q''\sim q'$ then $q'=\Phi(h(q))$ and $q''=\Psi(g(q'))$ for some $h\in \ol{H}_p$, $g\in \ol{H}_{\Phi(p)}$, and $\Phi,\Psi\in \mathsf{D}$. Then $q''=(\Psi\circ \Phi)((g' \circ h)(q))$, where $g'=\Phi^{-1}\circ g\circ \Phi\in \ol{H}_p$, and therefore $q''\sim q$.
\\

2. Let $L'$ denote a leaf of $\widehat{\fol}^\ell$. From (1), the intersection of $L'$ with $U_p$ is a union of orbits of $\ol{H}_p$. On the other hand, we claim that the intersection $L'\cap U_p$ consists of countably many orbits of $\ol{H}_p$, so that each connected component of such intersection must consists of a single $\ol{H}_p$-orbit. From the definition of $\widehat{\fol}^\ell$, it is enough to prove that the subgroup $\mathsf{D}_p\subset \mathsf{D}$ of diffeomorphisms fixing $p$ moves every $\ol{H}_p$-orbit in $L'\cap U_p$ to at most countably many orbits. For this, consider a piecewise smooth loop $\gamma:[0,1]\to L_p$ with $\gamma(0)=\gamma(1)=p$. Using linearized vector fields with $\gamma$ as integral curve, one can construct a continuous path $\Phi_t:[0,1]\to \mathsf{D}$ of diffeomorphisms such that $\Phi_0=id_U$ and $\Phi_t(p)=\gamma(t)$, as described in \cite[Cor. 7]{Mendes-Radeschi}. Fixing  some $\ol{H}_p$-orbit $\mathcal{O}$ in $L'\cap U_p$, its image $\Phi_1(\mathcal{O})$ is again some $\ol{H}_p$-orbit, which only depends on the class $[\gamma]\in \pi_1(L_p,p)$ and not on the actual path $\gamma$, nor on the specific choice of $\Phi_t$. This gives a map
\[
\partial: \pi_1(L_p,p)\to \{\ol{H}_p\textrm{-orbits in }L'\cap U_p\}
\]
This map admits a section, namely: for every orbit $\mathcal{O}'$ in $L'\cap U_p$, take a path $\gamma$ in $L'$ from a point in a (fixed) orbit $\mathcal{O}$ to a point in $\mathcal{O}'$. Under the projection $\pr:U\to B$, the composition $\pr\circ\gamma$ is a loop in $L_p$. The section of $\partial$ sends $\mathcal{O}'$ to $[\pr\circ\gamma]\in \pi_1(L_p,p)$. In particular, the map $\partial$ is surjective, and therefore the set of $\ol{H}_p$-orbits in $L'\cap U_p$ has at most the cardinality of $\pi_1(L_p,p)$, which is at most countable since $L_p$ is a manifold.
\end{proof}

As a corollary of Propositions \ref{P:prop-loc-closure} and \ref{P:local-form}, one gets the following:
\begin{corollary}\label{C:local-form}
Let $(U,\fol)$ be a singular Riemannian foliation as in Section \ref{S:setup}, let $\fol^\ell$ be its linearized foliation and $\widehat{\fol}^\ell$ the local closure. Then $\widehat{\fol}^\ell$ is a singular foliation with complete leaves. Moreover, around each point $p\in B$ there is a neighbourhood $W$ of $p$ in $B$ such that $(\pr^{-1}(W),\widehat{\fol}^\ell|_{\pr^{-1}(W)})$ is foliated diffeomorphic to a product
\[
(\DD^k\times \DD^{m-k}\times U_p,\DD^k\times\{\textrm{pts.}\}\times \{\textrm{orbits of }\ol{H}_p\In \OO(U_p)\}),
\]
which can be given the structure of a singular Riemannian foliation. 
%Finally, along $U_p$ the distributions $\mathcal{T}, \mathcal{N}$ and $\mathcal{K}$ are tangent to the first, second and third factor, respectively.
%\footnote{removed the last sentence here. I added it before we met, in order to prove that the linearized foliation is homogeneous. But now we have a better proof and we don't need this.}
\end{corollary}

Once it is shown that $\widehat{\fol}^\ell$ is also a transnormal system with respect to some metric, then by the corollary above it is globally a singular Riemannian foliation.

\section{A new metric}\label{S:proof}

Let $\mathcal{T},\,\mathcal{N},\,\mathcal{K}$ be the distributions as in the previous section. Clearly, one has $TU=\mathcal{T}\oplus\mathcal{N}\oplus\mathcal{K}$. 
%Moreover, recall that along each fiber $\pr^{-1}(p)$, the distribution $\mathcal{S}$ consists of the tangent spaces of the slice $S_p$.
%For any $p$, let $\metric_p$ be the flat metric on $S_p$ (and therefore on the spaces $\mathcal{S}_q$ for $q\in \pr^{-1}(p)$).

Define now the new metric $\metric^\ell$ on $U$, as the metric defined by the following properties:
%\footnote{slightly changed the wording here, but I can go back to the old version}
\begin{itemize}
\item $\mathcal{T}\oplus\mathcal{N}$ and $\mathcal{K}$ are orthogonal with respect to $\metric^\ell$.
\item $\metric^\ell|_{\mathcal{T}\oplus \mathcal{N}}=\pr^*\metric_\base$, where $\metric_\base$ denotes the restriction of the original metric on $\base$. In particular, $\mathcal{T}$ and $\mathcal{N}$ are also orthogonal to one another.
\item For any $q\in U_p$, recall that $\mathcal{K}_q=T_qU_p$, and define $\metric^\ell|_{\mathcal{K}_q}=\metric_p$ the flat metric on $U_p$ induced from $\exp_p:\nu_p\base\to U_p$.
\end{itemize}

These conditions characterize the metric $\metric^\ell$ uniquely. The most useful property of this metric is the following.

\begin{proposition}\label{P:new-metric}
The triples $(U,\metric^\ell,\fol^\ell)$ and $(U,\metric^\ell,\widehat{\fol}^\ell)$ are singular Riemannian foliations.
\end{proposition}
\begin{proof}
The arguments for $\fol^\ell$ and $\widehat{\fol}^\ell$ are ideantical, therefore we will only check the Proposition for $(U,\metric^\ell,\widehat{\fol}^\ell)$ (which is the only case we need for the Main Theorem anyway).
%\footnote{First sentence added}
 Moreover, the statement is local in nature, therefore it is enough to prove the statement on certain open sets covering the whole of $U$. For any point $p\in \base$, let $W$ denote a neighbourhood of $p$ in $\base$ and $\pr^{-1}(W)$ a neighbourhood of $p$ in $U$. We need to check that $(\pr^{-1}(W),\metric^\ell,\widehat{\fol}^\ell)$ is a singular Riemannian foliation. To prove this, we apply Proposition 2.14 of \cite{Alexandrino-desingularization} which states that it is enough to check two conditions:
\begin{enumerate}
\item $(\pr^{-1}(W),g',\widehat{\fol}^\ell)$ is a singular Riemannian foliation with respect to \emph{some} Riemannian metric $g'$.
\item For every stratum $\Sigma\subset \pr^{-1}(W)$ (i.e. union of leaves of the same dimension), the restriction of $\widehat{\fol}^\ell$ to $\Sigma$ is a (regular) Riemannian foliation. 
\end{enumerate}

The first condition is satisfied by Corollary \ref{C:local-form}. The second condition is equivalent to checking that, for every leaf $\widehat{L}_q$ of $\widehat{\fol}^\ell|_{\pr^{-1}(W)}$ and every basic vector field $X$ along $\widehat{L}_q$ tangent to the stratum through $\widehat{L}_q$, the norm $\|X\|_{\metric^\ell}$ is constant along $\widehat{L}_q$.

By definition of the metric $\metric^\ell$, the space $\nu \widehat{L}_q$ is given by $\mathcal{N}|_{\widehat{L}_q}\oplus (\nu \widehat{L}_q \cap \mathcal{K})$. From Proposition \ref{P:extension}, along $\widehat{L}_q$ the space $\mathcal{N}$ is spanned by linearized vector fields $X^\ell_i$, which are then $\metric^\ell$-basic (i.e., foliated \emph{and} $\metric^\ell$-orthogonal to the leaves). In particular, any basic vector field $\ol{X}$ along $\widehat{L}_q$ splits as a sum $\ol{X}=\ol{X}_1+\ol{X}_2$, where $\ol{X}_1$ is tangent to $\mathcal{N}$, $\ol{X}_2$ is tangent to $\mathcal{N}':=\nu \widehat{L}_q \cap \mathcal{K}$, and $\metric^\ell(\ol{X}_1,\ol{X}_2)=0$. Therefore, it is enough to check independently that for every basic vector field $\ol{X}$ along $\widehat{L}_q$, tangent to either $\mathcal{N}$ or $\mathcal{N}'$, the norm of $\ol{X}$ is constant along $\widehat{L}_q$.

If $\ol{X}$ is tangent to $\mathcal{N}$, then by the construction in Proposition \ref{P:extension} it projects to some basic vector field $X$ along $\widehat{L}_p\subset B$. Since  $(B,\metric_B,\fol|_B)$ is a Riemannian foliation, the norm $\|X\|_{\metric_B}$ is constant along $\widehat{L}_p$. By the construction of the metric $\metric^\ell$, one has $\|\ol{X}\|_{\metric^\ell}=\|X\|_{\metric_B}$ and, therefore, the norm of $\ol{X}$ is constant along $\widehat{L}_q$. 

If $\ol{X}$ is tangent to $\mathcal{N}'$, then it is tangent to any fiber $U_{p'}$, $p'\in \widehat{L}_p$. The restriction $\ol{X}|_{U_{p'}}$ is a basic vector field of $(U_{p'},\widehat{\fol}^\ell_{p'})$ along $\widehat{L}_{q}\cap U_{p'}$, and therefore the norm $\|\ol{X}|_{U_{p'}}\|_{\metric_{p'}}$ is locally constant along $\widehat{L}_{q}\cap U_{p'}$. By the construction of $\metric^\ell$, it follows that  $\|\ol{X}|_{U_{p'}}\|_{\metric^\ell}$ is also locally constant along each $\widehat{L}_{q}\cap U_{p'}$. However, given two points $p', p''\in \widehat{L}_p$, and a vertical, foliated vector field $V^\ell$ whose flow $\Phi$ moves $p'$ to $p''$, one also has that $\Phi$ moves $U_{p'}$ isometrically to $U_{p''}$, and $\ol{X}|_{U_{p'}}$ to $\ol{X}|_{U_{p''}}$. In particular, $\|\ol{X}|_{U_{p'}}\|_{\metric^\ell}=\|\ol{X}|_{U_{p'}}\|_{\metric_{p'}}$ does not really depend on the point $p'\in \widehat{L}_p$, and it is actually constant along the whole leaf $\widehat{L}_q$.
\end{proof}

With this in place, one can finally prove the Main Theorem:
\begin{proof}[Proof of the Main Theorem]
Let $U$ be an $\epsilon$-tubular neighbourhood around the closure $\ol{L}$ of a leaf $L\in \fol$. Letting $B=\ol{L}$, we are under the assumptions of Section \ref{S:setup}. In particular, it is possible to define the linearized foliation $\fol^\ell$ on $U$, its local closure $\widehat{\fol}^\ell$, and the metric $\metric^\ell$ as in Proposition \ref{P:new-metric}. It is clear by construction that $\widehat{\fol}^\ell|_{\ol{L}}=\fol|_{L}$ and that the closure of $\widehat{\fol}^\ell$ is contained in the closure of $\fol$. Moreover, by Corollary \ref{C:local-form} the foliation $(U,\metric^\ell,\widehat{\fol}^\ell)$ is, locally around each point, foliated diffeomorphic to the orbit like foliation $(\DD^k\times \DD^{m-k}\times U_p,\DD^k\times\{\textrm{pts.}\}\times \widehat{\fol}^\ell_p)$. By Proposition \ref{P:orbit-like-invariant}, the foliation $(U,\metric^\ell,\widehat{\fol}^\ell)$ is orbit like as well, and this concludes the proof.
\end{proof}

\bibliographystyle{amsplain}

\end{document}

%% file: 3distributions.tex
\begin{pspicture}(3,5)

%the frame of Sigma
 \psline[linewidth=0.5pt](0,0)(2,0.6) %Lower horiz line
 \psline[linewidth=0.5pt](0,0)(0,3.3) %Left vert line
 \psline[linewidth=0.5pt](0,3.3)(2,4) %Higher horiz line
 \psline[linewidth=0.5pt](2,0.6)(2,4) %Right vert line
 \rput(2.4,3.8){$B$}  

%Leaves in Sigma
 \psline[linewidth=0.5pt](1.5,0.45)(1.5,3.8) %Right leaf
 \psline[linewidth=0.5pt](0.5,0.15)(0.5,3.5) %Left leaf
 \psline[linewidth=1.5pt](1,0.3)(1,3.65) %Central leaf
 \rput(1.2,3.3){$L$}

%Distributions in \Sigma
%\psline[linewidth=1pt,linearc=.25]{->}(1,1.7)(1,2.7)
%\rput(1.2,2.9){${T\fol|_B}$}
%\psline[linewidth=1pt,linearc=.25]{->}(1,1.7)(2,2)
%\rput(1.75,2.2){${\nu\fol|_B}$}

%The Distributions
%\pspolygon*[linearc=0,linecolor=lightgray, fillstyle=solid, fillcolor=lightgray](2,1.25)(2.5,1.4)(3,1.25)(2.5,1.1)
%\pspolygon*[linearc=0,linecolor=lightgray, fillstyle=solid, fillcolor=lightgray](1.5,1.25)(2.5,1.55)(3.5,1.25)(2.5,0.95)

\psline[linewidth=1pt,linearc=.25]{->}(2.5,1.25)(2.5,2.25)
\rput(2.6,2.45){${\mathcal{T}}$}
\psline[linewidth=1pt,linearc=.25]{->}(2.5,1.25)(3.5,1.55)
\rput(3.25,1.75){${\mathcal{N}}$}
\psline[linewidth=1pt,linearc=.25]{->}(2.5,1.25)(3.5,0.95)
\rput(3.25,0.75){$\mathcal{K}$}

% The points p and q
\psdots(1,1.7)
\rput(0.8,1.5){$p$}

\psdots(2.5,1.25)
\rput(2.3,1.05){$q$}

\psline[linewidth=0.5pt](1,1.7)(3,1.1) %segment pq

%%%
%%% LINEARIZATION
%%%
%
%%the frame of Sigma
% \psline[linewidth=0.5pt](6,0)(8,0.6) %Lower horiz line
% \psline[linewidth=0.5pt](6,0)(6,3.3) %Left vert line
% \psline[linewidth=0.5pt](6,3.3)(8,4) %Higher horiz line
% \psline[linewidth=0.5pt](8,0.6)(8,4) %Right vert line
% \rput(2.4,3.8){$B$}  
%
%%Leaves in Sigma
% \psline[linewidth=0.5pt](7.5,0.45)(7.5,3.8) %Right leaf
% \psline[linewidth=0.5pt](6.5,0.15)(6.5,3.5) %Left leaf
% \psline[linewidth=1.5pt](7,0.3)(7,3.65) %Central leaf
% \rput(7.2,3.3){$L$} 
%
%%The Distributions
%%\pspolygon*[linearc=0,linecolor=lightgray, fillstyle=solid, fillcolor=lightgray](2,1.25)(2.5,1.4)(3,1.25)(2.5,1.1)
%%\pspolygon*[linearc=0,linecolor=lightgray, fillstyle=solid, fillcolor=lightgray](1.5,1.25)(2.5,1.55)(3.5,1.25)(2.5,0.95)
%
%\psline[linewidth=1pt,linearc=.25]{->}(8.5,1.25)(8.5,2.25)
%\rput(8.6,2.45){$\mathcal{T}$}
%\psline[linewidth=1pt,linearc=.25]{->}(8.5,1.25)(9.5,1.55)
%\rput(9.25,1.75){$\mathcal{N}$}
%\psline[linewidth=1pt,linearc=.25]{->}(8.5,1.25)(9.5,0.95)
%\rput(9.25,0.75){$\mathcal{K}$}
%
%% The points p and q
%\psdots(7,1.7)
%\rput(6.8,1.5){$p$}
%
%\psdots(8.5,1.25)
%\rput(8.3,1.05){$q$}
%
%\psline[linewidth=0.5pt](7,1.7)(9,1.1) %segment pq
%
\end{pspicture}